\documentclass[12pt,a4paper]{article}

\usepackage{amssymb, amsmath, amsthm}

\usepackage[cm]{fullpage}
\usepackage[english]{babel}
\usepackage[cp1250]{inputenc}
\usepackage[T1]{fontenc}
\usepackage[pdftex]{graphicx}
\usepackage{booktabs}

\newtheorem{df}{Definition}
\newtheorem{thm}{Theorem}
\newtheorem{lem}{Lemma}
\newtheorem{prop}{Proposition}

\title{Hopf bifurcation in a conceptual climate model with ice-albedo and precipitation-temperature feedbacks}
\author{\L ukasz P\l ociniczak\thanks{Faculty of Pure and Applied Mathematics, Wroc{\l}aw University of Science and Technology, Wyb. Wyspia{\'n}skiego 27, 50-370 Wroc{\l}aw, Poland}$\;^,$\footnote{Email: lukasz.plociniczak@pwr.edu.pl}}
\date{}

\begin{document}
\maketitle

\begin{abstract}
	
In this paper we analyse a dynamical system based on the so-called KCG (K\"all\'en, Crafoord, Ghil) conceptual climate model. This model describes an evolution of the globally averaged temperature and the average extent of the ice sheets. In the nondimensional form the equations can be simplified facilitating the subsequent analysis. We consider the limiting case of a stationary snow line for which the phase plane can be completely analysed and the type of each critical point can be determined. One of them can exhibit the Hopf bifurcation for which existence we find sufficient conditions. Those, in turn, have a straightforward physical meaning and indicate that the model predicts internal oscillations of the climate. Using the typical real-world values of appearing parameters we conclude that the obtained results are in the same ballpark as the conditions on our planet during the quaternary ice ages. Our analysis is a rigorous justification of a generalization of some previous results by KCG and other authors. \\

\end{abstract}

\section{Introduction}
The story of our planet is a fascinating one having many spectacular moments and interesting phenomena. Focusing just on the climate, the data shows that the Earth has undergone a series of prolonged periods of low temperature during which the ice sheets descended to somewhat low latitudes - the ice ages. Apart some very extreme events - such as the hypothetical Snowball Earth (see for ex. \cite{Pie11,Hyd00}) - the last series of glaciations during the Quaternary period showed a quasi-periodic behaviour. These oscillations were characterized by a slow build-up of ice sheets and then a rapid deglaciation (see Fig. \ref{fig:LisieckiRaymo}). This approximately regular recent behaviour of the climate is peculiar and they are several theories which try to explain it. The most popular one, founded by the research of Milutin Milankovitch \cite{Mil98} and later firmly established by the experimental findings of Hays, Imbrie, and Shackleton \cite{Hay76}, says that the astronomical variations in the Earth's orbit (axial tilt, eccentricity and precession) act as a pacemaker of the global temperature. 

\begin{figure}
	\centering
	\includegraphics[scale=1.4]{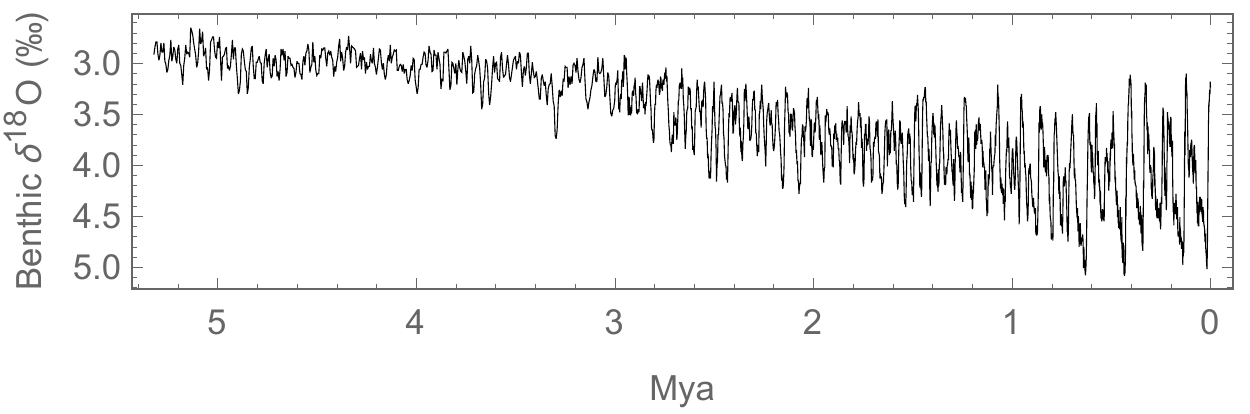}
	\includegraphics[scale=1.4]{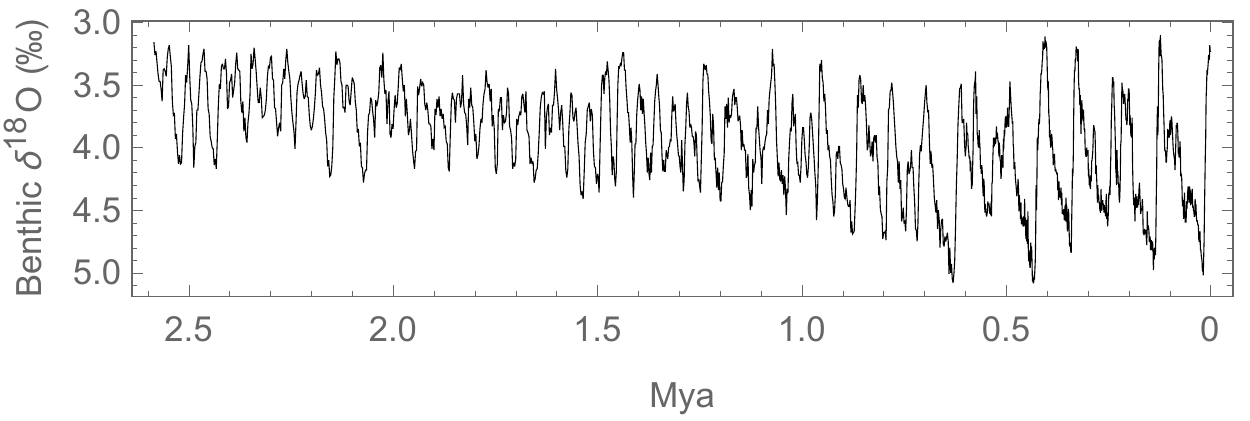}
	\caption{Time series of the benthic $\delta^{18}$O (a ratio $^18$O/$^16$O) as a proxy for temperature variations. During the warm periods the heavier isotope of the oxygen $^{18}$O evaporates from the oceans more readily and then is transported to higher latitudes where falls as precipitation. Therefore, smaller $\delta^{18}$O indicates higher temperature. The top plot shows the last 5.3 million years while the bottom plot is a magnification of the Pleistocene epoch (last 2.588 million years). The vertical axis is reversed so that the temperature increases upwards. Moreover, the time flows rightwards and the present is marked by 0. Data is taken from \cite{Lis05}.}
	\label{fig:LisieckiRaymo}
\end{figure}

Milankovitch theory is very attractive since the main frequencies of the astronomical forcing closely correspond to those observed in the geological data \cite{McG12,Ber88}. There are, however, certain problems with that simple explanation. First, the dominant period of the glacial oscillations during the last million years is about 100ka which closely corresponds to the period of orbital eccentricity changes. On the other hand, its can be easily calculated that the amplitude of that astronomical forcing is too small to drive the climate to the observed extent. Moreover, during the times before 1 Mya the dominant period of the temperature was approximately equal to 40ka. This peculiar bifurcation of the oscillation frequency cannot be ascribed to the Milankovitch theory (see \cite{Cla06}). As was suggested for example in \cite{Le83,Tzi06,Fow13,De13}, the linear astronomical forcing can be treated as a tuning (or pacemaker) mechanism for a nonlinear \emph{internal} oscillations of the climate. Some of the important ingredients of this mechanism are various positive and negative feedbacks from the ice sheets, precipitation, concentration of CO$_2$ and deep ocean circulation \cite{Hog08,Sal84,Pai01,Fow13}. Moreover, the aforementioned change of oscillation period, the so-called Mid-Pleistocene Transition, can be explained by a particular bifurcation \cite{Qui17,Sal84,Ash15}. 

The various climate feedbacks can be analysed by the use of conceptual climate models which differ from the complex Global Climate Models (GCMs) by focusing only on some specific averaged properties of the Earth system \cite{Cru12,Cla02}. These models are usually composed of several differential equations and can help to isolate the particular ingredients of the whole climate. To mention only a few examples, Saltzman and Maasch proposed and developed a physically based model describing glacial oscillations as a tectonically and astronomically forced limit cycle \cite{Sal84,Maa90}. The main variables of the model are total ice mass, CO$_2$ concentration, and the North Atlantic deep water. Their theory has later been summarized in a book \cite{Sal02}. The mathematical treatment of this model has been recently conducted in \cite{Eng17}. Furthermore, a more elaborated box model of the ocean has been constructed in \cite{Gil00}. Recently, an interesting dynamical model exhibiting a subcritical Hopf bifurcation has been introduced in \cite{Pai04} modelling the ice volume, Antarctic ice sheet extent and the atmospheric concentration of CO$_2$. Moreover, in \cite{Cru11} the so-called minimal model for the ice-ages has been introduced as a forced van der Pol oscillator. Interestingly, some conceptual climate models exhibit mixed mode oscillations and this case has been rigorously investigated in \cite{Rob15}. 

Many approaches to modelling climate dynamics start from formulating the net energy balance averaged over the planet. These so-called Energy Balance Models (EBMs) have been initiated by seminal papers of Budyko \cite{Bud69} and Sellers \cite{Sel69} and nowadays are known by their name. Their various generalizations applied to many problems in climate dynamics have been investigated in a series of papers by North and his collaborators (for a summary see \cite{Nor75,Nor81} and also \cite{Nor83}). There is also a recent and interesting rigorous study of the dynamics generated by Budyko-Sellers model supplemented by varying the ice line (extent of the ice sheet). Specific results can be found for example in \cite{McG12,McG14,Wid13,Wal16}. Moreover, in\cite{Fow13,Fow15} the ice energy balance notion has been supplemented by additional physical equations for the ice sheet movement and carbon cycle in an impressive dynamical system modelling these complex interactions. Finally, we mention the paper by K\"all\'en, Crafoord, Ghil \cite{Kal97} on which we base our reasoning. In that work the energy balance is coupled with ice sheet movement and a physically plausible limit cycle is numerically found. Further results on that model can be found in \cite{Ghi81,Ghi83}.

In what follows we generalize the KCG model by introducing arbitrary functional forms of the ice-albedo and precipitation-temperature feedbacks. In reality, those functions cannot be determined explicitly by any measurement. We show that assuming only the physically obvious monotonicity and boundedness we are able to prove the results observed by previous Authors for the specific choices of aforementioned functions. We then consider a simplified version of the snow line model and rigorously justify that certain simplifications can be conducted in the nondimensional form of the system. Furthermore, we classify the type of each stationary point that can arise in this general situation. In particular, under conditions which we precisely state, one of them can exhibit a Hopf bifurcation yielding a limit cycle with an explicitly known first Lypaunov coefficient. In Section 2 we derive the model and rescale it to the nondimensional form. Section 3 contains the results of our analysis. 

\section{Model derivation}
\subsection{Energy balance}
We begin with a derivation of the governing equations which are statements of the energy and mass conservation. The initial step in the model formulation is a Budyko-Sellers type of energy conservation. In our work we assume the so-called zero-dimensional approximation meaning that we are concerned with globally averaged energy balance over sufficiently long time periods. Schematically, the energy equation has the form
\begin{equation}
	c \frac{dT}{dt} = q_i - q_o,
\label{eqn:EnergyConservation}
\end{equation}
where $T$ is the globally averaged temperature, $c$ is the atmosphere heat capacity while $q_i$ and $q_o$ are, respectively, incoming shortwave and outgoing longwave radiative thermal fluxes. One can directly calculate that $c$ defined in the above formula is approximately equal to $m_a c_a / A_E$, with $m_a$ being the mass of the atmosphere, $c_a$ specific heat of air and $A_E$ Earth's area (see \cite{Fow13}).

Earth receives the energy in a high quality (mostly) shortwave solar radiation. Let $Q$ denote the solar constant, i.e. the mean amount of Sun's irradiance per unit area of a plane perpendicular to solar rays. Assuming parallel ray approximation, Earth receives the amount of energy per unit time equal to $\pi R_E^2 Q $ with $R_E$ being the radius of Earth. Since $4\pi R_E^2$ is the total Earth's surface area we have
\begin{equation}
	q_i = \frac{1}{4}(1-\alpha)Q,
\label{eqn:IncomingFlux}
\end{equation}
where $\alpha$ is the mean terrestrial albedo (the fraction of reflected to absorbed radiation). It has to be noted that by no means should $Q$ and $\alpha$ be regarded as spatially independent. For example, the time-average solar irradiance (so-called insolation) has a profound meridional distribution. This distribution can be calculated directly from the spherical geometry and has been done for example in \cite{Nad17}. One should point, however, that the horizontal distribution of heat by turbulent motion of the atmosphere takes place at a much faster time-scale than the one that interests us in our conceptual model (see \cite{Fow11}). Hence, it is not unreasonable to assume that the whole planet has a well-defined average uniform temperature. 

The aforementioned albedo $\alpha$ is a function of the underlying type of matter, for example ice reflects more light than a forest. In the present, when we globally average the albedo we obtain a value close to $0.3$. Our foregoing considerations will be based on separating the continental $\alpha_c$ and oceanic $\alpha_o$ parts of the albedo
\begin{equation}
	\alpha(T,l) = \gamma \alpha_c(l) + (1-\gamma) \alpha_o(T),
\end{equation}
where $l$ is the latitudinal extent of the ice sheets and $\gamma$ is a fraction of the area of continents to the whole surface of Earth. Of course $0\leq \alpha \leq 1$. This model has been proposed in \cite{Kal97}. If more ice covers the land, the albedo should become larger. Therefore, $\alpha_c=\alpha_c(l)$ will be later chosen to be an increasing function. On the other hand, the albedo of the ocean is usually taken to account for two extremes: the lack and the presence of sea ice. The formation of the sea ice is, in turn, strongly associated with the temperature: the colder the climate, the more sea ice is created increasing the oceanic albedo. For example, in \cite{Sel69,Kal97} the albedo of the ocean has been taken to be a piecewise linear function, while in \cite{Fow11} a hyperbolic tangent smoothing approximation has been used. As we will show below, the exact form of the oceanic albedo is not relevant for the quantitative dynamics of the climate and hence we will use an arbitrary sigmoid function.

The outgoing longwave radiation (OLR) is Earth's main mean of energy loss. Because the atmosphere contains many absorbers sensitive to the long electromagnetic waves (hence the Greenhouse effect), the precise quantitative description of the OLR is a complex task (see for ex. \cite{Pie10}). Staying at the conceptual level we will adopt a semi-empirical approach in determining $q_o$. Since the magnitude of OLR closely agrees with incoming solar radiation, the planet is in a energetic equilibrium. We could thus use the Stefan-Boltzmann's law that states that the outgoing flux is equal to $\sigma e^{-\Gamma} T^4$. Here, $\sigma$ is the Stefan-Boltzmann constant while $\Gamma$ is an empirical parameter modelling the radiation absorption in the atmosphere - the so-called greenhouse parameter. This approach had been taken in a number of works, for example in \cite{Fow11, Fow13}. We will, however, model the outgoing flux in a supposedly simpler way which dates back to Budyko's original work \cite{Bud69}. Mostly due to the short interval of relevant temperatures, say $250-320$K, the sensible choice is to propose the following linear dependence 
\begin{equation}
	q_o = A + B T,
\label{eqn:OutgoingFlux}
\end{equation}
where $A$ and $B$ are constants determined from the data (to be given below in Tab. \ref{tab:Parameters}). It has also been previously argued that $A$ can be related to the concentration of CO$_2$ in the atmosphere (see \cite{Wal13}). Finally, putting (\ref{eqn:IncomingFlux}) and (\ref{eqn:OutgoingFlux}) in (\ref{eqn:EnergyConservation}) we obtain the equation of energy conservation 
\begin{equation}
	c \frac{dT}{dt} = \frac{Q}{4}\left(1-\gamma \alpha_c(l) - (1-\gamma)\alpha_o(T)\right)-A-B T,
\label{eqn:EnergyConservationFull}
\end{equation}
where $\alpha_o$ is of the form (\ref{eqn:alphao}). Taking the $l$ to be constant one can immediately see that the above equation enjoys multiple steady states which forms the basis of many fascinating phenomena associated with climate dynamics (see \cite{Pai98,Dit18}). A careful examination of the corresponding bifurcation diagram shows hysteretic behaviour thanks to which the so-called tipping points of abrupt climate change are present \cite{Len11,Fow11}.

\subsection{Mass balance}
By the mass balance we understand the growth and retreat of ice sheets under the influence of the climate variation. We will derive an equation describing the average behaviour of those great masses of ice. Our reasoning is based on the ice sheet model proposed in \cite{Wee76} (but see also \cite{Oer80,Kal97}). This is a simplified model based on the plastic flow and isostasy. The more accurate and elaborated models based on Glen's law are surveyed in \cite{Van13,Oer84,Fow97}. 

\begin{figure}
	\centering
	\includegraphics[scale=1]{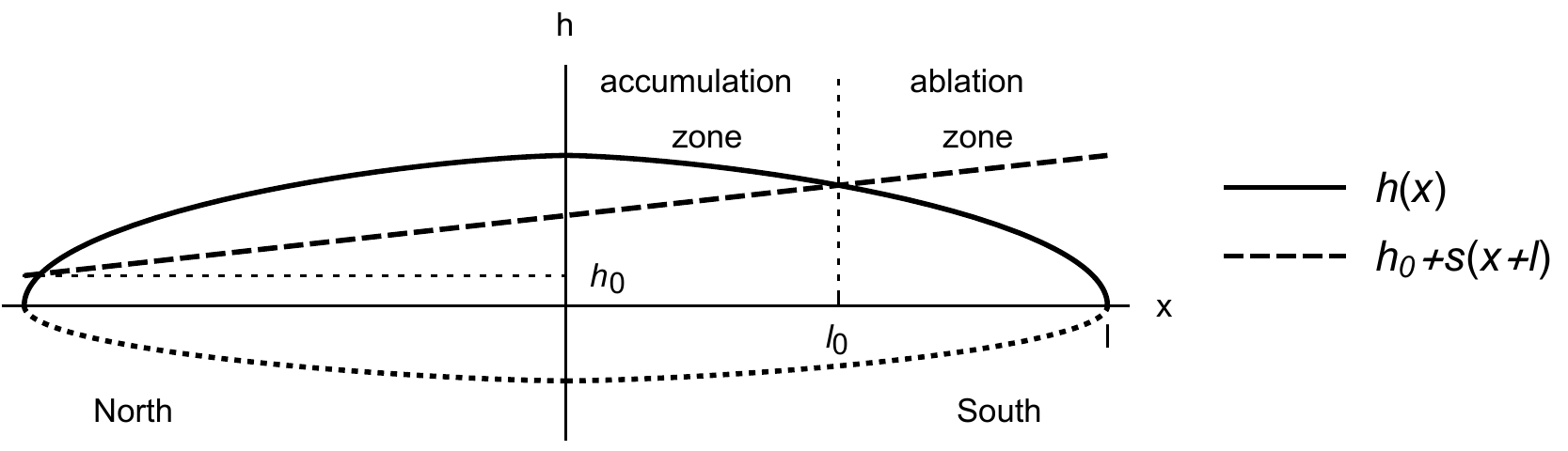}
	\caption{Schematic view of an ice sheet. Adapted from \cite{Wee76}.}
	\label{fig:IceSheet}
\end{figure}

The overall picture of the ice sheet is presented on Fig. \ref{fig:IceSheet}. The ice sheet is assumed to be zonally symmetric, i.e. have an essentially two-dimensional latitudinal profile. Moreover, we orient the $x$-axis pointing southward with origin at the ice sheet's half-width $l$. The northward rim $x=-l$ is located at the Artic Ocean where the ice shelves form. Assume that the mass of ice of height $h(x)$ above the sea level rests on the isostatically depressed bedrock. Following \cite{Wee76} we can assume that the rock density is three times the density of ice. Then, the buoyancy law states that the total ice thickness is equal to $\frac{3}{2}h(x)$. 

The ice sheet moves according to the plastic flow law, i.e. the whole mass behaves as a plastic fluid experiencing the yield stress $\tau_0$. Since the normal stresses are hydrostatic, in the equilibrium we must have
\begin{equation}
	\tau_0 = \frac{3}{2}\rho_i g h \left|\frac{dh}{dx}\right|,
\end{equation}
where $\rho_i$ is the density of ice. This equation can be immediately integrated giving the parabolic profile
\begin{equation}
	h(x) = \sqrt{\frac{4\tau_0 l}{3\rho_i g}\left(1-\frac{|x|}{l}\right)}=H\sqrt{l}\sqrt{1-\frac{|x|}{l}}, \quad H:=\sqrt{\frac{4\tau_0}{3\rho_i g}}.
\label{eqn:SheetProfile}
\end{equation}
The typical value of the yield stress and the corresponding height scale is given in Tab. \ref{tab:Parameters}.

The dynamics of the ice sheet is governed by an interplay of snow accumulation and melting. We assume that the northern part of the ice sheet is in equilibrium with its southern counterpart. Therefore, the mass balance of the whole sheet is governed by the ablation and accumulation in the region $[0,l]$. The amount of snowfall is closely associated with the temperature via the so-called snow line. The temperature falls with the height and thus there is a well-defined $0^\circ C$ isotherm under which the snow, if fallen, will melt. In the cited works this isotherm has been taken to be linear, i.e.
\begin{equation}
	h_{iso}(x) = h_0 + s(x+l),
\label{eqn:Snowline}
\end{equation}	
where $h_0$ is the height at the Arctic ocean while $s$ is the slope parameter. The northern height of the isotherm depends generally on the temperature or astronomical forcing. This has been taken into account in the literature. For example in \cite{Wee76,Fow13} Authors correlated the variation in $h_0$ with the Milankovitch oscillations while in \cite{Kal97} a linear dependence on the global temperature has been prescribed.

The snowline position $x=l_0$ defining the part of the sheet which is nourished by the snowfall is defined as a solution of $h(x) = h_{iso}(x)$. The mass balance can then be written as
\begin{equation}
	\frac{d}{dt}\int_0^l \frac{3}{2}h(x) dx = a l_0 - m (l-l_0),
\label{eqn:MassBalance1}
\end{equation}
where $a$ and $m$ are respectively accumulation and melting rates. Now, if we assume that the ice sheet can grow or retreat, that is $l=l(t)$, we can obtain
\begin{equation}
	\frac{d}{dt}\int_0^l \frac{3}{2}h(x) dx = \frac{3}{2}H\frac{d}{dt}\int_0^l \sqrt{l}\sqrt{1-\frac{x}{l}} dx = \frac{3}{2}H \sqrt{l} \; \frac{dl}{dt}.
\end{equation}
Furthermore, the value of $l_0$ can be found by equating (\ref{eqn:SheetProfile}) and (\ref{eqn:Snowline})
\begin{equation}
	h_0 + s(l_0+l) = H \sqrt{l}\sqrt{1-\frac{l_0}{l}},
\label{eqn:SnowLine}
\end{equation}
which is a quadratic equation with the meaningful solution
\begin{equation}
\label{eqn:l0}
	l_0 = \frac{H^2}{s^2}\left[-\left(\frac{h_0 s}{H^2}+\frac{s^2}{H^2}l+\frac{1}{2}\right)+\sqrt{\frac{h_0 s}{H^2}+2 \frac{s^2}{H^2}l+\frac{1}{4}}\right],
\end{equation}
which together with (\ref{eqn:MassBalance1}) gives the second dynamic equation
\begin{equation}
	\frac{3}{2}H \sqrt{l}\frac{dl}{dt} = (a+m) \frac{H^2}{s^2}\left[-\left(\frac{h_0 s}{H^2}+\frac{s^2}{H^2}l+\frac{1}{2}\right)+\sqrt{\frac{h_0 s}{H^2}+2 \frac{s^2}{H^2}l+\frac{1}{4}}\right] - ml \quad \text{for } l_0 \geq 0.
\label{eqn:MassConservationFull}
\end{equation}
This is a rather complicated nonlinear equation but we will shortly see that an appropriate scaling and approximation will yield its accurate simplification. Also notice that nowhere in the above discussion we have mentioned the response of the bedrock to the evolution of the ice sheet. The moving mass will produce a delayed feedback from the lithosphere and hence, might provide an essential ingredient of the dynamics. This requires adding an additional equation to our system and a simplified version of it have been introduced in \cite{Ghi81,Le83}. In this paper, however, we will consider only the instantaneous adjustment of the bedrock leaving the more general problem for the future work.

Two additions are necessary for the well-posedness of the model. First, it may happen that the whole southern side of the ice sheet is in the ablation zone. This can happen for sufficiently large sheets or by raising the $0^\circ$C isotherm. In this case the whole mass of ice becomes stagnant and our derivation is not valid (equation (\ref{eqn:SnowLine}) does not have a positive solution). Authors of \cite{Wee76} and \cite{Fow13} propose to assume that we can take $l_0 = 0$ and then, the mass balance can be written as
\begin{equation}
	\frac{3}{2}H \sqrt{l} \; \frac{dl}{dt} = - m l \quad \text{for} \quad l_0 < 0.
\label{eqn:MassBalancel0Negative}
\end{equation}
Another situation of (\ref{eqn:SnowLine}) failing to have a solution is when the negative snowline elevation $h_0$ becomes smaller than $-sL$. Then, the snow will accumulate on the ground and if $l=0$ the ice sheet will nucleate. For a sufficiently high snowfall rates $a$ (or sufficiently small $-h_0/s$) the reasonable approximation is to assume (see \cite{Wee76}) that $l = -\frac{h_0}{2s}$ and 
\begin{equation}
	\frac{3}{2}H \sqrt{l} \; \frac{dl}{dt} = - \frac{ah_0}{2s} \quad \text{for} \quad l<-\frac{h_0}{2s},
\label{eqn:MassBalanceh0Negative}
\end{equation}
while for $l > -h_0/s$ the main equation (\ref{eqn:MassConservationFull}) still holds.

\subsection{Nondimensionalization}
We are finally in position to nondimensionalize our model, i.e. equations (\ref{eqn:EnergyConservationFull}) and (\ref{eqn:MassConservationFull}). The summary of all of the physical quantities used in this model is given in Tab. \ref{tab:Parameters}, while a reference to all symbols appearing in the text is collected in Tab. \ref{tab:Symbols}.

\begin{table}
\centering
\begin{tabular}{ccc}
	\toprule
	Symbol & Meaning & Typical value \\
	\midrule
	$Q$ & Solar constant & 1361 W m$^{-2}$ \\
	$\gamma$ & Continent to ocean area ratio & $0.3$ \\
	$A$ & Budyko constant in OLR flux & -267.96 W m$^{-2}$\\
	$B$ & Budyko constant in OLR flux & 1.74 W m$^{-2}$ K$^{-1}$\\
	$\tau_0$ & Ice sheet yield stress & 0.3$\times 10^5$ Pa \\
	$\rho_i$ & Ice density & 0.92$\times 10^3$ kg m$^{-3}$ \\
	$H$ & Ice sheet height scale & 2.1 m$^\frac{1}{2}$ \\
	$s$ & $0^\circ$C Isotherm slope & 0.4 $\times 10^{-3}$\\
	$h_0$ & $0^\circ$C Isotherm height over the Arctic Ocean & 1.2 km\\
	$\epsilon$ & $0^\circ$C Isotherm height over the Arctic Ocean (nondimensional) & 0.1\\
	$T^*$ & Temperature time scale & 195.55 K \\
	$L^*$ & Ice sheet extent scale & 2.7$\times 10^4$km\\
	$t^*$ & Time scale & 33.2$\times 10^3$ years\\
	$\mu$ & Ratio of ocean and ice sheet specific heats & bifurcation parameter \\
	$\beta$ & Nondimensional parameter in Budyko-Sellers model & 0.79\\
	$\alpha_1$, $\alpha_2$ & Parameters of the continental albedo & 0.25 and 4, respectively \\
	$\alpha_-$, $\alpha_+$ & Limits of oceanic albedo & 0.85 and 0.25, respectively\\
	$\xi_-$, $\xi_+$ & Limits of the ratio of accumulation and ablation & 0.1 and 0.5, respectively\\
	$\theta_\alpha$, $\Delta \alpha$ & Translation and steepness parameters for oceanic albedo & 1.4 and 0.015, respectively \\
	$\theta_\xi$, $\Delta \xi$ & Translation and steepness parameters for $\xi$ & 1.43 and 0.0027, respectively \\
	\bottomrule
\end{tabular}
\caption{Typical values of all of the physical parameters used in the manuscript. The data is based on \cite{Kal97,Fow13}.}
\label{tab:Parameters}
\end{table}

\begin{table}
	\centering
	\begin{tabular}{ccc}
		\toprule
		Symbol & Meaning & Definition \\
		\midrule
		$T$, $\theta$ & Temperature (dimensional and nondimensional) & (\ref{eqn:EnergyConservation}), (\ref{eqn:Scales})\\
		$l$, $\lambda$ & Ice sheet extent (dimensional and nondimensional) & (\ref{eqn:MassBalance1}), (\ref{eqn:Scales}) \\
		$t$, $\tau$ & Time (dimensional and nondimensional) & (\ref{eqn:Scales})\\
		$q_{i,o}$ & Incoming and outgoing fluxes & (\ref{eqn:EnergyConservation})\\
		$h$ & Ice sheet profile & (\ref{eqn:SheetProfile}) \\
		$h_{iso}$ & $0^\circ$C isotherm & (\ref{eqn:Snowline})\\
		$l_0$ & Boundary between the accumulation and ablation zones & (\ref{eqn:l0}) \\
		$\lambda_0$ & Position of $l_0$ measured in terms of $\lambda$ & (\ref{eqn:Lambda0})\\
		$\sigma$ & Generic sigmoid function & Def. \ref{def:Sigmoid}\\
		$\alpha_{c,o}$ & Continental and oceanic albedo & (\ref{eqn:alphao}), (\ref{eqn:alphac}) \\
		$\xi$ & Ratio of accumulation to ablation rate & (\ref{eqn:MassConservationScaled})\\
		$F$, $G$ & Coordinates of the vector field & (\ref{eqn:MainSystem})\\
		$f$, $g$ & $\theta$- and $\lambda$-nullclines & (\ref{eqn:Nullclines}) \\ 
		\bottomrule
	\end{tabular}
	\caption{Symbols appearing in the paper along with their meaning and location of definition.}
	\label{tab:Symbols}
\end{table}

First, we introduce the following scaling
\begin{equation}
\label{eqn:Scales}
	T (t) = T^* \theta(\tau), \quad l(t) = L^* \lambda(\tau), \quad t = t^* \tau,
\end{equation}
where starred letters denote the appropriate scales while those in Greek correspond to the new dimensionless variables. The choice for $T^*$ is straightforward. Since, as we mentioned before, the Earth is almost in thermal equilibrium where incoming solar radiation is balanced with OLR we can take
\begin{equation}
	T^* = \frac{Q}{4 B},
\label{eqn:TempScale}
\end{equation}
which, using our typical data from Tab. \ref{tab:Parameters} is $T^* = 195.55$K. Then, the energy equation becomes
\begin{equation}
	\frac{c}{B t^*} \frac{d\theta}{d\tau} = 1+\beta-\gamma\alpha_c(\lambda)-(1-\gamma)\alpha_o(\theta) - \theta, \quad \beta:=-\frac{4A}{Q} > 0,
\end{equation}
where we have retained the same letters to denote (already dimensionless) albedo functions. Next, it is convenient to introduce the dimensionless snow line position in the following way
\begin{equation}
\label{eqn:Lambda00}
	l_0(t) = L^* \lambda(\tau) \lambda_0(\tau).
\end{equation}
In this way $\lambda_0$ measures the position of the boundary of the ice sheet's accumulation level in terms of the total ice extent $\lambda$. We would also like to choose the ice sheet length scale $L^*$ according to the $0^\circ$C isotherm. A quick look at (\ref{eqn:SnowLine}) lets us choose
\begin{equation}
	L^* = \frac{H^2}{s^2}, \quad \epsilon = \frac{s h_0}{H^2}.
\label{eqn:LengthScale}
\end{equation}
In this way we have
\begin{equation}
	\lambda_0(\lambda) = \frac{1}{\lambda} \left[-\left(\epsilon+\lambda+\frac{1}{2}\right)+\sqrt{\epsilon+2\lambda+\frac{1}{4}}\right],
\label{eqn:Lambda0}
\end{equation}
which implies the nondimensional form of the mass balance equation
\begin{equation}
	\frac{d\lambda}{d\tau} = \sqrt{\lambda}\left(\left(1+\xi\right)\lambda_0-1\right) \quad \text{for} \quad \lambda_0 \geq 0.
\label{eqn:MassConservationScaled}
\end{equation}
where we have denoted the ratio of accumulation to ablation by $\xi$ and chosen the ice sheet time scale
\begin{equation}
	t^* = \frac{3}{2} \frac{H^2}{m s}.
\label{eqn:TimeScale}
\end{equation} 
Using the typical parameters from Tab. \ref{tab:Parameters} we can estimate that $t^* = 33.2\times 10^3$ years, which roughly corresponds to the usual glacial oscillation time scale. Moreover, $L^*=2.7\times 10^4$ km and $\epsilon = 0.1$. Notice that $\epsilon$ can be thought as a small parameter. Having done all the scalings we can now formulate the final version of the temperature equation
\begin{equation}
	\frac{1}{\mu}\frac{d\theta}{d\tau} =  1+\beta-\gamma\alpha_c(\lambda)-(1-\gamma)\alpha_o(\theta) - \theta, 
\end{equation}
where
\begin{equation}
	\mu := \frac{3}{2} \frac{B H^2}{m s c},
\end{equation}
is the free parameter measuring the ratio of ocean and ice sheet specific heats. We will also take this as our bifurcation parameter (what has also been done in \cite{Kal97}). This is because it is relatively easy to show that its variation produces a limit cycle via the Hopf bifurcation. And these in turn are of fundamental interest in climate science.

As for the modifications which take into accout $l_0 < 0 $ and $h_0 < 0$ the nondimensionalization is straightforward. First, equation (\ref{eqn:MassBalancel0Negative}) translates into
\begin{equation}
	\frac{d\lambda}{d\tau} = -\sqrt{\lambda} \quad \text{for} \quad \lambda_0 < 0,
\end{equation}
while (\ref{eqn:MassBalanceh0Negative}) becomes
\begin{equation}
	\frac{d\lambda}{d\tau} = -\frac{\xi}{2\sqrt{\lambda}} \epsilon \quad \text{for} \quad \lambda < -\frac{\epsilon}{2}.
\end{equation} 
If $\lambda \geq -\frac{\epsilon}{2}$ the main equation (\ref{eqn:MassConservationScaled}) is valid.

To close the model we have to prescribe the albedos $\alpha_{c,o}$ and the melting to snowfall ratio $\xi$. The continental albedo should increase with the ice sheet extent. The simplest model of that is, of course, a linear dependence
\begin{equation}
\alpha_c(\lambda) = \alpha_1 + \alpha_2 \lambda,
\label{eqn:alphac}
\end{equation}
where $\alpha_i$ are empirical constants. We have to keep in mind that in order of the above formula to produce a bounded albedo, the ice sheet extent has to be kept appropriately limited. The value $\alpha_1$ is that of a clean ground while $\alpha_2$ is determined from the continent's complete ice cover. If the ice is maximally developed, the continental albedo should become almost equal to that of the pure ice. Hence, we have the estimate
\begin{equation}
	\alpha_2 \geq \alpha_{cmax}-\alpha_1,
\end{equation}
where we have assumed that the maximal ice sheet extent is attained for $\lambda=1$. Thus, when we use the typical values from Tab. \ref{tab:Parameters} we obtain that $\alpha_1 \geq 4$. 

The oceanic albedo is a decreasing function of the temperature and, as we mentioned before, many authors proposed different versions of it: from stepwise functions to its continuous approximations. As we will show, the dynamics does not depend on the particular form of the albedo. 
\begin{df}
\label{def:Sigmoid}
	A function $\sigma:\mathbb{R}\rightarrow\mathbb{R}$ is called \textbf{sigmoid} if it is bounded and differentiable with a non-negative derivative. As a normalization one can take $\lim\limits_{x\rightarrow\pm\infty} \sigma(x) = \pm 1$.
\end{df}  
Therefore, it is reasonable to propose the following form of the oceanic albedo
\begin{equation}
\alpha_o(\theta) = \frac{1}{2}\left(\alpha_+ +\alpha_- + (\alpha_+-\alpha_-)\;\sigma_{\alpha}\left(\frac{\theta-\theta_\alpha}{\Delta \alpha}\right)\right), \quad \alpha_- \geq \alpha_+,
\label{eqn:alphao}
\end{equation}
where $0\leq\alpha_\pm\leq 1$ are limits of $\alpha_o(\theta)$ for $\theta\rightarrow \pm\infty$, $\theta_\alpha$ is the translation and $\Delta \alpha$ being the steepness parameter. The limiting case when $\Delta \alpha\rightarrow 0^+$ corresponds to the discontinuous jump from $\alpha_-$ to $\alpha_+$ at $\theta=\theta_\alpha$. This form has been chosen for example in \cite{McG12}.

In \cite{Kal97} Authors argue that $\xi=\xi(\theta)$ is an increasing function determined mostly by the snowfall rate. When the temperatures are low the oceans are mostly covered with ice. The evaporation is small, and hence the precipitation rate is low. Therefore, we can take
\begin{equation}
	\xi(\theta) = \frac{1}{2}\left(\xi_+ +\xi_- + (\xi_+-\xi_-)\;\sigma_{\xi}\left(\frac{\theta-\theta_\xi}{\Delta \xi}\right)\right), \quad \xi_- \leq \xi_+,
\label{eqn:xi}
\end{equation}
where $\sigma_\xi$ is an arbitrary sigmoid function and $0\leq\xi_\pm\leq 1$ are respective limits of $\xi(\theta)$ when $\theta\rightarrow\pm\infty$. This concludes the derivation of our general model.

\section{Analysis}
\subsection{Simplification}
So far we have stayed at the full generality and did not make any simplifications corresponding to the particular configuration of the ice sheet and the temperature. It is crucial to notice that the $\lambda$-coordinate for the most interesting stationary point of the model is generally small. First, notice that by taking the value of $h_0$ estimated in \cite{Wee76,Fow13} to be maximally of order of $1.2$km, the reasonable magnitudes of $\tau_0 = 0.3 \times 10^5$ Pa and $s = 0.4\times 10^{-3}$ we have $\epsilon \approx 0.1$ which can be thought of as a small quantity. First, we will estimate the magnitude of the stationary value of the ice sheet extent. 

Notice that only the case of $\lambda_0 \geq 0$ yields an interesting dynamics and therefore, in the following we will only focus on the analysis of (\ref{eqn:MassConservationScaled}). It is relatively easy to obtain the estimates of the magnitude of the $\lambda$-nullcline on which the possible stationary points can lie. 

\begin{prop}
	Assume that
	\begin{equation}
		-\frac{2+\xi}{2\xi}\leq \epsilon < \frac{1}{4}\frac{\xi}{2+\xi},
	\end{equation} 
	where $\xi=\xi(\theta)$. Then the $\lambda$-nullcline of (\ref{eqn:MassConservationScaled}) decomposes into two branches $\lambda_{1,2}=\lambda_{1,2}(\theta)$ for which $0\leq \lambda_1 < \lambda_2$. Moreover, the following holds for $\epsilon >0$ 
	\begin{equation}
	\begin{split}
		\left(1+\frac{1}{\xi}\right)\epsilon^2&\leq \lambda_1(\theta) \leq 4\left(1+\frac{1}{\xi}\right)\epsilon^2, \\ 2\left(1-\frac{1}{2+\xi}\right)\epsilon &\leq \frac{\xi(1+\xi)}{(2+\xi)^2}-\lambda_2(\theta) \leq 3\left(1-\frac{1}{2+\xi}\right)\epsilon,
	\end{split}
	\end{equation}
	while for $\epsilon \leq 0$ we have
	\begin{equation}
	\begin{split}
		\left(1+\frac{1}{\xi}\right)\epsilon^2+2\left(1+\frac{1}{\xi}\right)\left(1+\frac{2}{\xi}\right)\epsilon^3 &\leq \lambda_1(\theta) \leq \left(1+\frac{1}{\xi}\right)\epsilon^2,\\
		-\left(1-\frac{1}{2+\xi}\right)\epsilon &\leq \lambda_2(\theta)-\frac{\xi(1+\xi)}{(2+\xi)^2}\leq -2\left(1-\frac{1}{2+\xi}\right)\epsilon.
	\end{split}
	\end{equation}
\end{prop}
\begin{proof}
At the beginning assume that $\epsilon > 0$. Observe that the $\lambda$-nullcline of (\ref{eqn:MassConservationScaled}) with a help of (\ref{eqn:Lambda0}) can be determined from the equation
\begin{equation}
	\frac{1}{\lambda} \left[-\left(\epsilon+\lambda+\frac{1}{2}\right)+\sqrt{\epsilon+2\lambda+\frac{1}{4}}\right] = \zeta, \quad \zeta := \frac{1}{1+\xi}.
\end{equation}
Inspecting the left-hand side of the above we can notice that there exist exactly two solutions of the above if and only if the maximum of $\lambda_0(\lambda)$ is greater than $\zeta$. This maximum can easily be calculated to yield
\begin{equation}
	\lambda_0(\lambda_{max}) = \frac{1-4\epsilon}{1+4\epsilon} \quad \text{where} \quad \lambda_{max} = \frac{1}{2}\epsilon(1+4\epsilon).
\end{equation}
Hence, two branches exist if and only if $\lambda_0(\lambda_{max}) > \zeta$, that is to say $\epsilon < \frac{1}{4}\frac{1-\zeta}{1+\zeta}$ which is our assumption. 

Furthermore, since the equation defining $\lambda_1$ and $\lambda_2$ is essentially a quadratic, its solutions can be readily obtained in a closed, yet cluttered, form
\begin{equation}
	\lambda_{1,2} = \frac{1}{2(1+\zeta)^2}\left[1-\zeta-2(1+\zeta)\epsilon\mp\sqrt{(1-\zeta)^2-4(1-\zeta^2)\epsilon}\right],
\end{equation}
which after some manipulation can be written as
\begin{equation}
	\lambda_{1,2} = \frac{1-\zeta}{2(1+\zeta)^2}\left[1-2\frac{1+\zeta}{1-\zeta}\epsilon\mp\sqrt{1-4\frac{1+\zeta}{1-\zeta}\epsilon}\right].
\label{eqn:LambdaCrit}
\end{equation}
The next step is to estimate the square root in order to extract the leading order behaviour when $\epsilon\rightarrow 0^+$. First, from the Taylor expansion we have
\begin{equation}
	\sqrt{1-4\frac{1+\zeta}{1-\zeta}\epsilon} \leq 1-2\frac{1+\zeta}{1-\zeta}\epsilon-2\left(\frac{1+\zeta}{1-\zeta}\right)^2\epsilon^2.
\end{equation}
From this we immediately obtain
\begin{equation}
	\lambda_1 \geq \frac{1-\zeta}{2(1+\zeta)^2}\left[1-2\frac{1+\zeta}{1-\zeta}\epsilon-1+2\frac{1+\zeta}{1-\zeta}\epsilon+2\left(\frac{1+\zeta}{1-\zeta}\right)^2\epsilon^2\right] = \frac{\epsilon^2}{1-\zeta}.
\end{equation}
Similarly, 
\begin{equation}
	\lambda_2 \leq \frac{1-\zeta}{2(1+\zeta)^2}\left[1-2\frac{1+\zeta}{1-\zeta}\epsilon+1-2\frac{1+\zeta}{1-\zeta}\epsilon\right]=\frac{1-\zeta}{(1+\zeta)^2}-\frac{2}{1+\zeta}\epsilon.
\end{equation}
To find the other estimates notice that because of the convexity of the square root we have (or by some elementary manipulations)
\begin{equation}
	\sqrt{1-4\frac{1+\zeta}{1-\zeta}\epsilon} \geq 1-4\frac{1+\zeta}{1-\zeta}\epsilon.
\end{equation}
Therefore,
\begin{equation}
	\lambda_2 \geq \frac{1-\zeta}{2(1+\zeta)^2} \left[2-6\frac{1+\zeta}{1-\zeta}\epsilon\right] = \frac{1-\zeta}{(1+\zeta)^2}-\frac{3}{1+\zeta}\epsilon.
\end{equation}
Finally, by the analysis of the auxiliary function $1-x -\sqrt{1-2x}-2x^2$ we can derive the last essential inequality
\begin{equation}
	\sqrt{1-4\frac{1+\zeta}{1-\zeta}\epsilon} \geq 1-2\frac{1+\zeta}{1-\zeta}\epsilon-8\left(\frac{1+\zeta}{1-\zeta}\right)^2\epsilon^2,
\end{equation}
which gives us the last estimate for $\epsilon \geq 0$. 

The case of $\epsilon < 0$ is analogous. The maximum of $\lambda_0$ is always equal to $1$ for $\epsilon < 0$ and hence, both branches exist then and we can use the estimates obtained from the Taylor series. The first branch $\lambda_1$ can be estimated with a use of
\begin{equation}
	2\left(\frac{1+\zeta}{1-\zeta}\right)^2\epsilon^2+4\left(\frac{1+\zeta}{1-\zeta}\right)^3\epsilon^3\leq 1-2\frac{1+\zeta}{1-\zeta}\epsilon-\sqrt{1-4\frac{1+\zeta}{1-\zeta}\epsilon} \leq 2\left(\frac{1+\zeta}{1-\zeta}\right)^2\epsilon^2, \quad 
\end{equation}
for $-\frac{1}{2}\dfrac{1+\zeta}{1-\zeta}\leq \epsilon \leq 0$. The second solution $\lambda_2$ can be bounded by
\begin{equation}
	1\leq \sqrt{1-4\frac{1+\zeta}{1-\zeta}\epsilon} \leq 1-2\frac{1+\zeta}{1-\zeta}\epsilon, \quad \epsilon \leq 0.
\end{equation} 
This along with (\ref{eqn:LambdaCrit}) ends the proof.
\end{proof}

Immediately we can see that 
\begin{equation}
	\lambda_1 = O(\epsilon^2), \quad \lambda_2 = \frac{\xi(1+\xi)}{(2+\xi)^2}+O(\epsilon) \quad \text{as} \quad \epsilon\rightarrow 0,
\end{equation}
and hence the branch converges to zero very fast with $\epsilon\rightarrow 0$ while the larger is near a positive value. Henceforth, it seems that $\lambda_2$ is the more essential branch of the mass balance equation on which a possible critical point can lie. Moreover, because $0\leq \xi\leq 1$ we have
\begin{equation}
	\frac{\xi(1+\xi)}{(2+\xi)^2} \leq \frac{\xi_+(1+\xi_+)}{(2+\xi_+)^2} \leq \frac{2}{9},
\end{equation}
since the left-hand side function is increasing. Therefore, any critical point, if exists, has $\lambda$ always close or smaller than $\frac{2}{9}$ (for values from Tab. \ref{tab:Parameters} we have $\xi_+=0.5$ and thus the upper bound is $0.12$). Because of that we can also consider it to be a small quantity. 

Now we are in position to make the final simplification of our governing equations. We want to utilize the fact that both $\epsilon$ and $\lambda$ are small quantities. In order to simplify the expression for $\lambda_0$ we write
\begin{equation}
	\lambda_0(\lambda) = \frac{1}{\lambda} \left[\lambda-\left(\epsilon+2\lambda+\frac{1}{2}\right)+\sqrt{\epsilon+2\lambda+\frac{1}{4}}\right].
\end{equation}
Next, since $\epsilon+2\lambda$ is small we can expand the above square root into Taylor series to obtain
\begin{equation}
	\lambda_0(\lambda) = \frac{1}{\lambda} \left[\lambda-(\epsilon+2\lambda)^2+2(\epsilon+2\lambda)^3+...\right]\approx 1-4\epsilon-4\lambda-\frac{\epsilon^2}{\lambda},
\end{equation}
where we have retained only two first terms of the above expansion. We can use this as an approximation of (\ref{eqn:MassConservationScaled}) but we will make yet another simplification by taking $\epsilon\rightarrow 0$. This is motivated by the above observation concerning the $\lambda$-nullcline branches $\lambda_{1,2}$. Finally, we arrive at our climate model which we repeat here for clarity. 
\begin{equation}
\left\{
\begin{array}{l}
	\dfrac{d\theta}{d\tau} =  \mu\left[1+\beta-\gamma\left(\alpha_1+\alpha_2 \lambda\right)-(1-\gamma)\alpha_o(\theta) - \theta\right]=:F(\theta,\lambda), \vspace{4pt} \\
	\dfrac{d\lambda}{d\tau} = \sqrt{\lambda}\left(\left(1+\xi(\theta)\right)\left(1-4\lambda\right)-1\right)=:G(\theta,\lambda),
\end{array}
\right.
\quad \text{for} \quad \theta> 0, \quad 0 < \lambda \leq \frac{1}{4},
\label{eqn:MainSystem}
\end{equation}
where $\alpha_o(\theta)$, $\xi(\theta)$ are respectively given by (\ref{eqn:alphao}) and (\ref{eqn:xi}) while all parameters $\mu$, $\beta$, $\gamma$ are positive. We will see that despite the apparent oversimplification of vanishing $\epsilon$, this dynamical system possesses an interesting oscillatory behaviour which is in accord with the observations in paleoclimatology. Therefore, even the simplest coupling of realistic energy and mass equations produces self-sustained internal climate oscillations.

A note about validity of the model (\ref{eqn:MainSystem}) is in order. First of all, as $\epsilon\rightarrow 0$ the smaller branch of the $\lambda$-nullcline disappears. The model should thus be valid only in the vicinity of the critical point lying on the $\lambda_2$ branch. On the other hand, the linear form of the continental albedo (\ref{eqn:alphac}) requires at most moderate variations of the ice sheet extent. Taking into account above arguments, we conclude that our simplifications are justified provided we do not allow for large excursions of $\lambda$.

\subsection{The phase plane}
Since $\lambda$ in (\ref{eqn:MainSystem}) appears only linearly in both of the given equations, we can immediately write the exact form of the nullclines (denoted by $f$ and $g$)
\begin{equation}
\begin{split}
	&\theta-\text{nullcline}: \quad \lambda =\frac{1}{\alpha_2}\left[ \frac{1}{\gamma}\left(1+\beta-(1-\gamma)\alpha_o(\theta)-\theta\right)-\alpha_1\right]=:f(\theta),\\
	&\lambda-\text{nullcline}: \quad \lambda = \frac{1}{4} \frac{\xi(\theta)}{1+\xi(\theta)}=:g(\theta).
\end{split}
\label{eqn:Nullclines}
\end{equation}
An exemplary plot of the phase plane is depicted on Fig. \ref{fig:PhasePlane}. We can quickly compute the respective derivatives 
\begin{equation}
	f'(\theta) = -\frac{1}{\gamma\alpha_2}\left((1-\gamma)\alpha_o'(\theta)+1\right), \quad g'(\theta) = \frac{1}{4}\frac{\xi'(\theta)}{\left(1+\xi(\theta)\right)^2}.
\label{eqn:NullclinesDerivatives}
\end{equation}
Observe that $g'$ is always positive while $f'$ can change its sign depending on whether the oceanic albedo $\alpha_o$ has a sufficiently steep gradient. For the real-world parameters $f$ has two local extrema and we denote them by $\theta_m$ and $\theta_M$. We can see that $f'$ is positive for $\theta\in(\theta_m,\theta_M)$ and negative otherwise.

\begin{figure}
	\centering
	\includegraphics[scale=0.85]{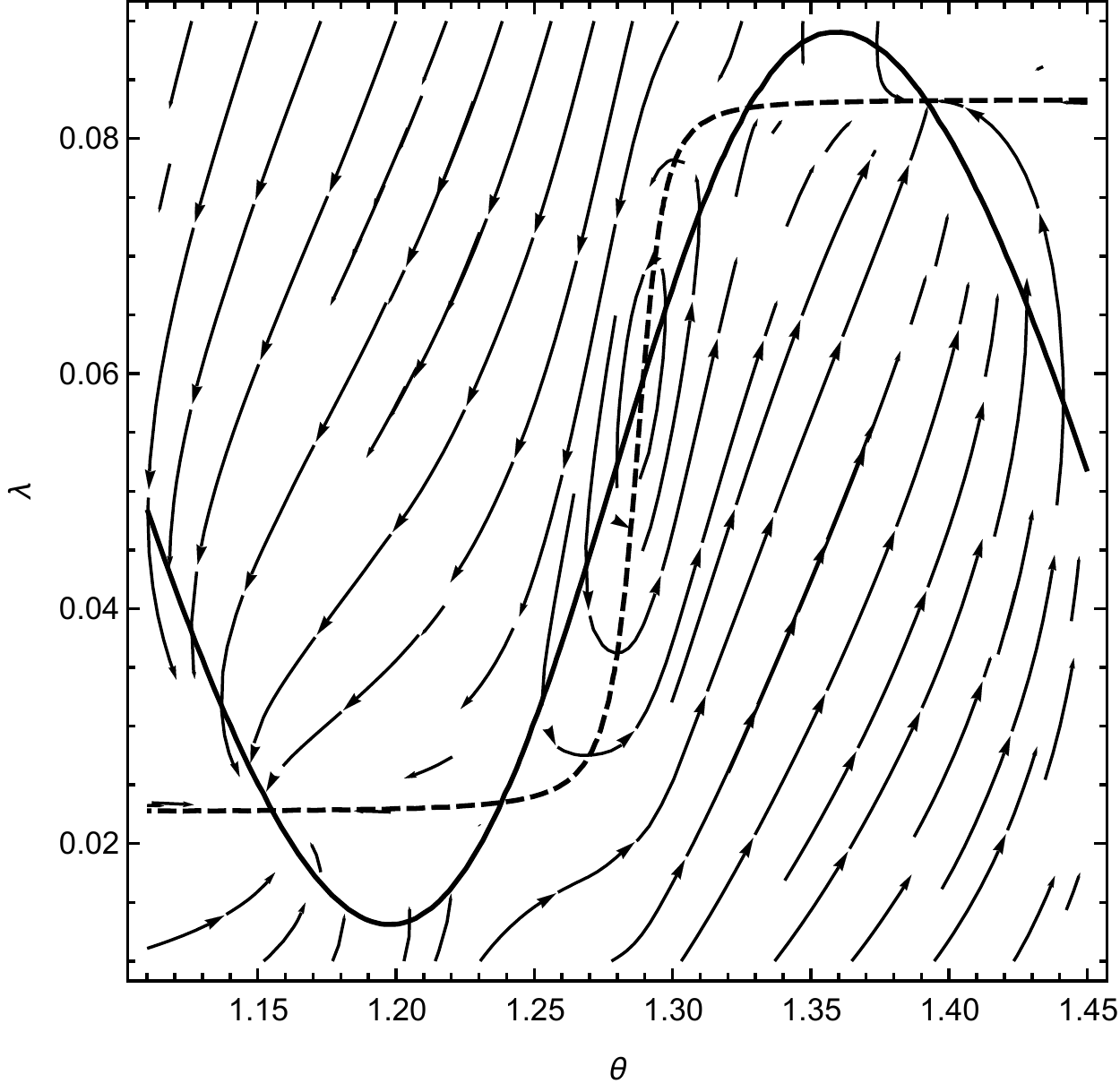}
	\caption{An exemplary phase plane along with nullclines $f$ (solid line) and $g$ (dashed line). The chosen parameters are not necessarily realistic and have been chosen to exemplify the qualitative features of the phase plane. In particular, $\alpha_0=0.25$, $\alpha_1=3.2$, $\alpha_+=0.25$, $\alpha_-=0.85$, $\theta_c=1.27$, $\Delta \alpha = 0.12$, $\xi_c=1.29$, $\Delta \xi = 0.01$. }
	\label{fig:PhasePlane}
\end{figure}

Since $\alpha_o(\theta)$ and $\xi(\theta)$ are both constructed with the sigmoid functions we can see that there can exist from one to five critical points of our system. With a simple geometrical reasoning we can distinguish several cases of their existence. 

\begin{prop}
	Let $n$ be a number of critical points of (\ref{eqn:MainSystem}) and denote $g_\pm := \lim\limits_{\theta\rightarrow \pm\infty} g(\theta)$. If $\Delta\alpha$ is small enough for $\alpha_o'(\theta) = (1-\gamma)^{-1}$ to have two zeros $\theta_m$ and $\theta_M$, then
	\begin{itemize}
		\item $n=1$ if $g_+ \leq f(\theta_m)$ or $g_- \geq f(\theta_M)$ or $\left(g(\theta_m) < f(\theta_m) \text{ and } g(\theta_M) > f(\theta_M)\right)$.
		\item $n\geq 3$ if $f(\theta_m) < g_-, g_+ < f(\theta_M)$.
		\item $n=5$ if and only if $f(\theta_m) < g_-, g_+ < f(\theta_M)$ and there exists a point $\theta_c\in(\theta_m,\theta_M)$ such that $f(\theta_c)=g(\theta_c)$ and $f'(\theta_c)<g'(\theta_c)$.
	\end{itemize}
	In the opposite case, we have $n=1$. 
\end{prop}
Since the above can be justified by considering each case and plotting the nullclines, we omit the proof. Because of a high number of various parameters the result is stated in terms of the values of nullclines at readily computable points. This makes the presentation clear. Note also that in the above observation we include only the most meaningful situations. Of course, one can arrange the parameters in a way to produce an even number of critical points but this would require having the nullclines tangent at some points. In reality, this situation is highly improbable however, we will return to this case in Proposition \ref{prop:Tangent}. 

The (local) stability of the critical points can be categorized according to the gradients of $f$ and $g$. Notice also that the creation and destruction of various stationary points depends on the parameters present in $\xi$ and $\alpha_o$ (for ex. $\Delta T$). Varying those can produce critical points through the saddle-node and pitchfork bifurcations. The precise conditions for them to occur are intrinsic to the specific nature of $\sigma$ functions used in construction of $\xi$ and $\alpha_o$. We will treat all parameters except $\mu$ to be fixed and consider all possible bifurcations. 

First, we give a simple result concerning the linearisation.

\begin{lem}
	Let $(\theta_c,\lambda_c)$ be one of the stationary points of (\ref{eqn:MainSystem}) and by $J$ denote its Jacobi linearisation matrix. Then
	\begin{equation}
		J = \left(
	\begin{matrix}
		\mu\alpha_2 \gamma f'_c & -\mu\alpha_2\gamma \vspace{4pt}\\
		\dfrac{\xi_c}{\sqrt{\lambda_c}}g'_c & -\dfrac{\xi_c}{\sqrt{\lambda_c}},
	\end{matrix}
	\right)
	\label{eqn:Jacobian}
	\end{equation}
	where we have used the short-hand notation $\xi_c:=\xi(\theta_c)$, $f'_c=f'(\theta_c)$ and $g'_c:=g'(\theta_c)$. Moreover, the eigenvalues of $J$ can be written as
	\begin{equation}
		r_{1,2} = \frac{1}{2}\left(\mu\alpha_2\gamma f'_c-\frac{\xi_c}{\sqrt{\lambda_c}}\pm \sqrt{\left(\mu\alpha_2\gamma f'_c-\frac{\xi_c}{\sqrt{\lambda_c}}\right)^2-4\mu\alpha_2\gamma \frac{\xi_c}{\sqrt{\lambda_c}}\left(g'_c-f'_c\right)}\right).
	\label{eqn:JacobianEigenvalues}
	\end{equation}
\end{lem}
\begin{proof}
This is just a straightforward calculation. We start from collecting the formulas for partial derivatives of the vector field $(F,G)$ defined in (\ref{eqn:MainSystem})
\begin{equation}
\begin{split}
	\frac{\partial F}{\partial \theta} &= -\left((1-\gamma)\alpha'_o(\theta)-1\right),\quad \frac{\partial F}{\partial \lambda} = -\mu\alpha_2\gamma, \\
	\frac{\partial G}{\partial \theta} &= \sqrt{\lambda}(1-4\lambda)\xi'(\theta),\quad \frac{\partial G}{\partial \lambda} = \frac{1}{2\sqrt{\lambda}}\left(\xi(\theta)-12\lambda(1+\xi(\theta))\right).
\label{eqn:PartialDerivatives}
\end{split}
\end{equation}
Since on the $\lambda-$nullcline (\ref{eqn:Nullclines}), in particular for the stationary point, we have $\lambda=g(\theta)$ the above partial derivatives of $G$ at $(\theta_c,\lambda_c)$ reduce to
\begin{equation}
	\frac{\partial G}{\partial \theta} = \frac{\lambda_c}{\sqrt{\lambda_c}}\frac{\xi'_c}{1+\xi_c}=\frac{\xi_c}{\sqrt{\lambda_c}}g'_c,\quad \frac{\partial G}{\partial \lambda} = \frac{1}{2\sqrt{\lambda_c}}\left(\xi_c-3\frac{\xi_c}{1+\xi_c}(1+\xi_c)\right)=-\frac{\xi_c}{\sqrt{\lambda_c}}.
\end{equation}
Moreover, the formula for the eigenvalues follows from the fact that
\begin{equation}
	r_{1,2} = \frac{1}{2}\left(\text{tr}J\pm\sqrt{\left(\text{tr}J\right)^2-4\det J}\right).
\end{equation}
This ends the proof.
\end{proof}
Notice that knowing (\ref{eqn:PartialDerivatives}) we can easily find the direction of the vector field $(F,G)$ since it vanishes on the nullclines and is monotone in their neighbourhood. 

Right now we are in position to subsequently classify the stability of all types of the stationary points that can arise in our system. At every critical point the nullclines $f$ and $g$ can cross transversalily or be tangent to each other. The transversality is the crucial property and can happen in two ways: either $f'>g'$ or otherwise. This classifies the particular cases of stability.

\begin{prop}
	\label{thm:Stable}
	Let $(\theta_c,\lambda_c)$ be any critical point of (\ref{eqn:MainSystem}) for which $f'(\theta_c)<0$. Then it is locally
	\begin{itemize}
		\item a stable node for $\mu\in(0,\mu_1] \cup [\mu_2,\infty)$,
		\item a stable focus for $\mu\in(\mu_1,\mu_2)$,
	\end{itemize}
	where
	\begin{equation}
		\mu_{1,2} := \frac{\xi_c}{\alpha_2\gamma \left(f'_c\right)^2 \sqrt{\lambda_c}}\left(2g'_c-f'_c\pm\sqrt{g'_c\left(g'_c-f'_c\right)}\right)>0.
	\label{eqn:Mu12}
	\end{equation}
\end{prop}
\begin{proof}
Since $f'_c<0$ we immediately have $\text{tr} J = \mu \alpha_2 \gamma f_c' - \frac{\xi_c}{\sqrt{\lambda_c}} < 0$ and thus the first term in both the eigenvalues (\ref{eqn:JacobianEigenvalues}) is negative. Moreover, because $g$ is always increasing we have $f'_c < g'_c$. Therefore, the quantity under the square root (\ref{eqn:JacobianEigenvalues}) can become negative for a suitable choice of $\mu > 0$. 

Immediately we notice that if the eigenvalues are conjugate complex the critical point is a stable focus since, as we mentioned before, $\text{Re} \; r_{1,2} < 0$. Next, observe that if the eigenvalues are real they have to be negative. The one with the minus sign is obviously smaller than zero while the other can be transformed into
\begin{equation}
	r_1 = \frac{2\mu\alpha_2\gamma \frac{\xi_c}{\sqrt{\lambda_c}}\left(g'_c-f'_c\right)}{\mu\alpha_2\gamma f'_c-\frac{\xi_c}{\sqrt{\lambda_c}}-\sqrt{\left(\mu\alpha_2\gamma f'_c-\frac{\xi_c}{\sqrt{\lambda_c}}\right)^2-4\mu\alpha_2\gamma \frac{\xi_c}{\sqrt{\lambda_c}}\left(g'_c-f'_c\right)}} < 0.
\label{eqn:R2}
\end{equation}
Therefore $(\theta_c,\lambda_c)$ is a stable node. In any case, the critical point is locally stable.

We are left with finding the precise conditions for the transition from focus to the node to occur. According to (\ref{eqn:JacobianEigenvalues}) we have to check whether the squared term chenges its sign. It is a quadratic equation for $\mu$ and can be written as
\begin{equation}
	\left(\alpha_2\gamma f'_c\right)^2 \mu^2 - 2 \alpha_2\gamma\frac{\xi_c}{\sqrt{\lambda_c}}\left(2g'-f'\right)\mu + \frac{\xi_c^2}{\lambda_c}.
\end{equation}
From the Vieta's formulas we immediately know that the zeros of the above are both positive. They can be easily calculated to yield (\ref{eqn:Mu12}).
\end{proof}

When we return to Fig. \ref{fig:PhasePlane} we notice that for the two outermost critical points we have $f'_c<0$ and hence, these are locally stable. The physical interpretation of this climate state is that we have either low temperature with a very small ice sheet or high temperature with a very large ice sheet. This situation is set up by the precipitation-temperature feedback and probably have to be corrected by allowing the snow line to vary with the temperature, i.e. $\epsilon=\epsilon(\theta)$. Since our simple model is tailored to investigate the other case we conclude that the two outermost critical points could not give a sufficient representation of the climate state. 

Now, we move to the most important and physical case where $f$ is increasing. 

\begin{thm}
	Let $(\theta_c,\lambda_c)$ be a critical point of (\ref{eqn:MainSystem}) for which $f'_c > 0$. Then, the following holds.
	\begin{itemize}
		\item When $g'_c < f'_c$ the critical point is a saddle for all all $\mu>0$.
		\item When $g'_c > f'_c$ the Hopf bifurcation takes place as $\mu$ passes
		\begin{equation}
			\mu_0 := \frac{\xi_c}{\alpha_2 \gamma f'_c \sqrt{\lambda_c}}.
		\label{eqn:Mu0}
		\end{equation}
		In this bifurcation a limit cycle is created with an approximate angular frequency 
		\begin{equation}
			\omega_0 = \frac{\xi_c}{\sqrt{\lambda_c}}\sqrt{\frac{g'_c}{f'_c}-1}.
		\label{eqn:Omega0}
		\end{equation}
		Its stability can be determined from the sign of the first Lyapunov coefficient $l_1(\mu_0)$ which, for clarity of presentation, is given in the Appendix (bifurcation is supercritical for $l_1(\mu_0)<0$ and subcritical otherwise).

		Moreover, the critical point locally is a 
		\begin{itemize}
			\item stable node for $\mu\in(0,\mu_1]$,
			\item stable focus for $\mu\in(\mu_1,\mu_0)$,
			\item unstable focus for $\mu\in(\mu_0,\mu_2)$,
			\item unstable node for $\mu\in[\mu_2,\infty)$,
		\end{itemize}
	where $\mu_{1,2}$ are defined in (\ref{eqn:Mu12}).
	\end{itemize}
\end{thm}
\begin{proof}
We start from the case of $g'_c - f'_c<0$. Then, from (\ref{eqn:JacobianEigenvalues}) we can see that the term under the square root is always positive which implies that the eigenvalues are real for every $\mu > 0$. Assume that $\mu$ is such that $\text{tr} J \leq 0$, then the eigenvalue with a minus sign is negative while for the second we have a formula (\ref{eqn:R2}) which shows that it has to be positive. Hence, the eigenvalues have opposite signs and the critical point is a saddle. The case $\text{tr} J > 0$ is analogous. 

Now, assume that $g'_c - f'_c > 0$ and notice that when determining the sign of the term under the square root in (\ref{eqn:JacobianEigenvalues}) we arrive at exactly the same situation as before in Theorem \ref{thm:Stable}. Hence, the critical point is a node for $\mu\in(0,\mu_1] \cup [\mu_2,\infty)$ and a focus for $\mu\in(\mu_1,\mu_2)$. The difference from this previous case is that its stability can change when $\mu$ increases. 

When $\mu=\mu_0$ the trace of the Jacobian matrix vanishes and the eigenvalues are purely imaginary. From the monotonicity of the quadratic we thus have $\mu_0 \in (\mu_1,\mu_2)$ and 
\begin{equation}
	\text{Re}\; r_{1,2}(\mu_0) = 0, \quad \text{Im}\;r_{1,2}(\mu_0) = \pm i\omega_0,
\end{equation}
where $\omega_0$ is defined in (\ref{eqn:Omega0}). Moreover,
\begin{equation}
	\left.\frac{d}{d\mu}\right|_{\mu=\mu_0} \text{Re} \; r_{1,2} = \frac{1}{2} \alpha_2 \gamma f_c' > 0,
\end{equation}
hence the critical points cross the imaginary axis with nonzero speed. Therefore, by the Hopf bifurcation theorem \cite{Guc13,Kuz13} we conclude that a limit cycle is born when $\mu\rightarrow \mu_0$. The stability $\mu$-intervals can be obtained essentially in the same way as before with the use of (\ref{eqn:R2}).

The last part of the theorem to prove is the form of the Lyapunov coefficient $l_1(\mu_0)$. First, we will transform (\ref{eqn:MainSystem}) into a standard form which helps us to utilize an explicit formula for $l_1(\mu_0)$. First, since at $\mu=\mu_0$ the Jacobian matrix has a conjugate pair of purely imaginary eigenvalues, we have
\begin{equation}
	J \textbf{q} = i \omega_0 \textbf{q},
\end{equation}
where the eigenvector has a form
\begin{equation}
	\textbf{q} = \left(
	\begin{array}{c}
		\frac{1}{g'_c}\left(1+i\sqrt{\dfrac{g'_c}{f'_c}-1}\right) \vspace{4pt}\\
		1
	\end{array}
	\right).
\end{equation}
Next, define the two transition matrices by
\begin{equation}
	S=\left(
	\begin{array}{cc}
		\textbf{q} & \overline{\textbf{q}}
	\end{array}\right),
	\quad
	T=\frac{1}{\sqrt{2}}\left(
	\begin{array}{cc}
	1 & i \\
	1 & -i
	\end{array}\right).
\end{equation}
If we now introduce a new coordinate set by
\begin{equation}
	\left(
	\begin{array}{c}
		\psi \\
		\kappa
	\end{array}\right) = T^{-1} S^{-1} 
	\left(
	\begin{array}{c}
		\theta-\theta_c \\
		\lambda-\lambda_c
	\end{array}
	\right)
\end{equation}
then (\ref{eqn:MainSystem}) transforms into
\begin{equation}
	\frac{d}{dt}
	\left(
	\begin{array}{c}
		\psi \\
		\kappa
	\end{array}\right) =  
	\left(
	\begin{array}{cc}
		0 & -\omega_0 \\
		\omega_0 & 0
	\end{array}
	\right)
	\left(
	\begin{array}{c}
		\psi \\
		\kappa
	\end{array}
	\right) + 
	\left(
	\begin{array}{c}
		P \\
		Q
	\end{array}
	\right),
\end{equation}
where 
\begin{equation}
\begin{split}
	\left(
	\begin{array}{c}
	P(\psi,\kappa) \\
	Q(\psi,\kappa)
	\end{array}
	\right) &= 
	\left(
	\begin{array}{c}
		G\left(\theta\lambda\right)+\omega_0 \kappa \\
		\left(g'_c/f'_c-1\right)^{-1/2}\left[G\left(\theta,\lambda\right)- \frac{\xi_c}{\alpha_2\gamma\sqrt{\lambda_c}}\frac{g'_c}{f'_c}F\left(\theta,\lambda\right)\right]-\omega_0 \psi
	\end{array} 
	\right) \\
	&\text{where } (\theta,\lambda)=\left(\theta_c+(g'_c)^{-1}\left(\psi-\kappa\left(g'_c/f'_c-1\right)^{1/2}\right),\lambda_c+\psi\right).
\end{split}
\end{equation}
Then, the first Lyapunov coefficient can be calculated via the formula (\cite{Guc13}, (3.4.29))
\begin{equation}
	\begin{split}
	l_1(\mu_0) &= \frac{1}{8\omega_0}(P_{\psi\psi\psi}+P_{\psi\kappa\kappa}+Q_{\psi\psi\kappa}+Q_{\kappa\kappa\kappa}) \\
	&+\frac{1}{8\omega_0^2}\left[P_{\psi\kappa}(P_{\psi\psi}+P_{\kappa\kappa}) -Q_{\psi\kappa}(Q_{\psi\psi}+Q_{\kappa\kappa})-P_{\psi\psi}Q_{\psi\psi}+P_{\kappa\kappa}Q_{\kappa\kappa}\right],
	\end{split}
\label{eqn:LyapunovFormula}
\end{equation}
where subscripts denote partial derivatives evaluated at $(\psi,\kappa)=(0,0)$. A straightforward calculation leads to (\ref{eqn:LyapunovCoefficient}).
\end{proof}
An exemplary plot of the limit cycle is presented on Fig. \ref{fig:TempIce} where we have used the parameters from Tab. \ref{tab:Parameters}. During our simulations we have noticed that although it is easy to control the position of the critical point, the stability of the limit cycle is much more difficult to anticipate. This is of course due to the rather complex form of (\ref{eqn:LyapunovCoefficient}) which involves several competitive terms. However, due to the sigmoid nature of $\xi$ and $\alpha_o$ their higher derivatives are very singular near their inflection point and thus, the presence of third derivatives in (\ref{eqn:LyapunovCoefficient}) can dominate the behaviour of $l_1(\mu_0)$. The precise analysis of the Lyapunov coefficient is one of the objectives of our future work.

\begin{figure}
	\centering
	\includegraphics[scale=0.85, trim = 100 0 0 0]{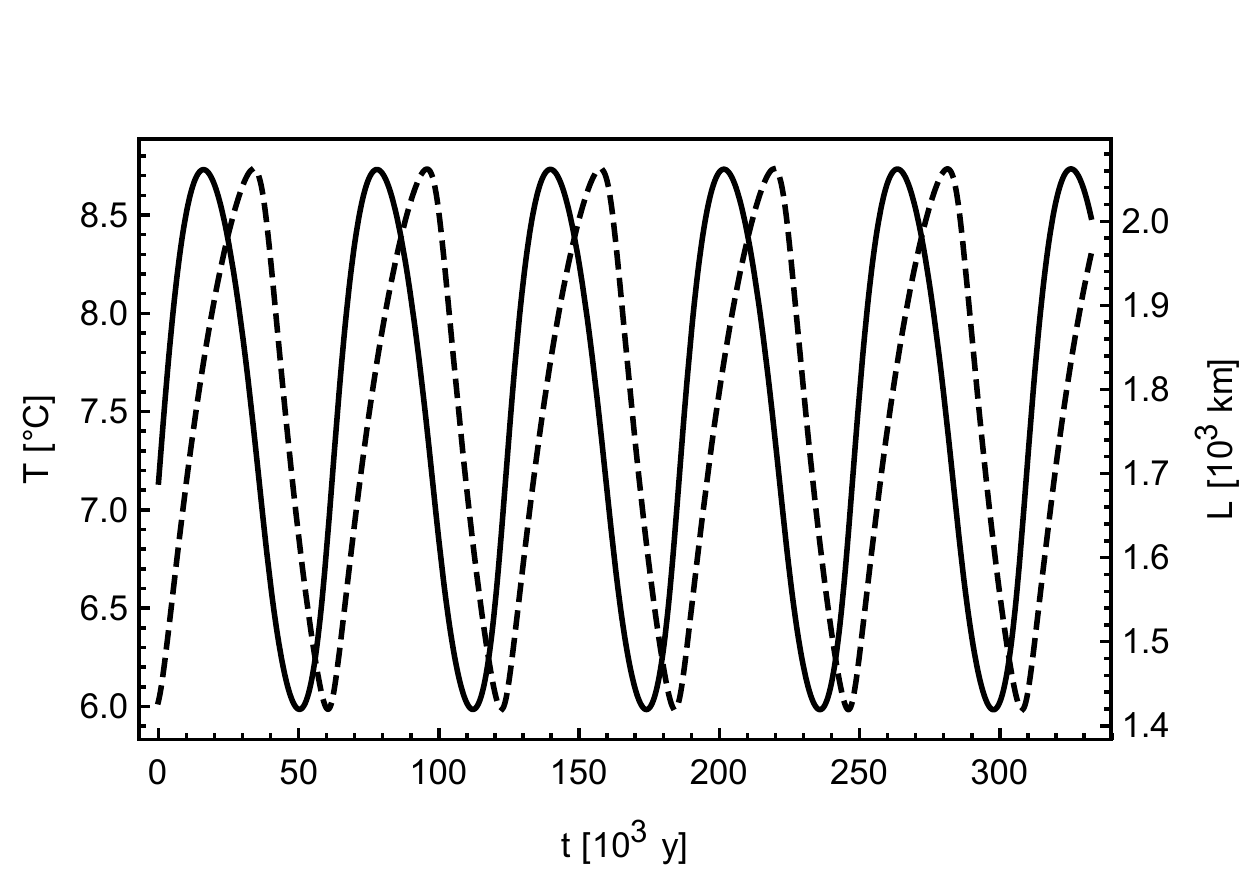}
	\includegraphics[scale=0.65, trim = 0 0 80 0]{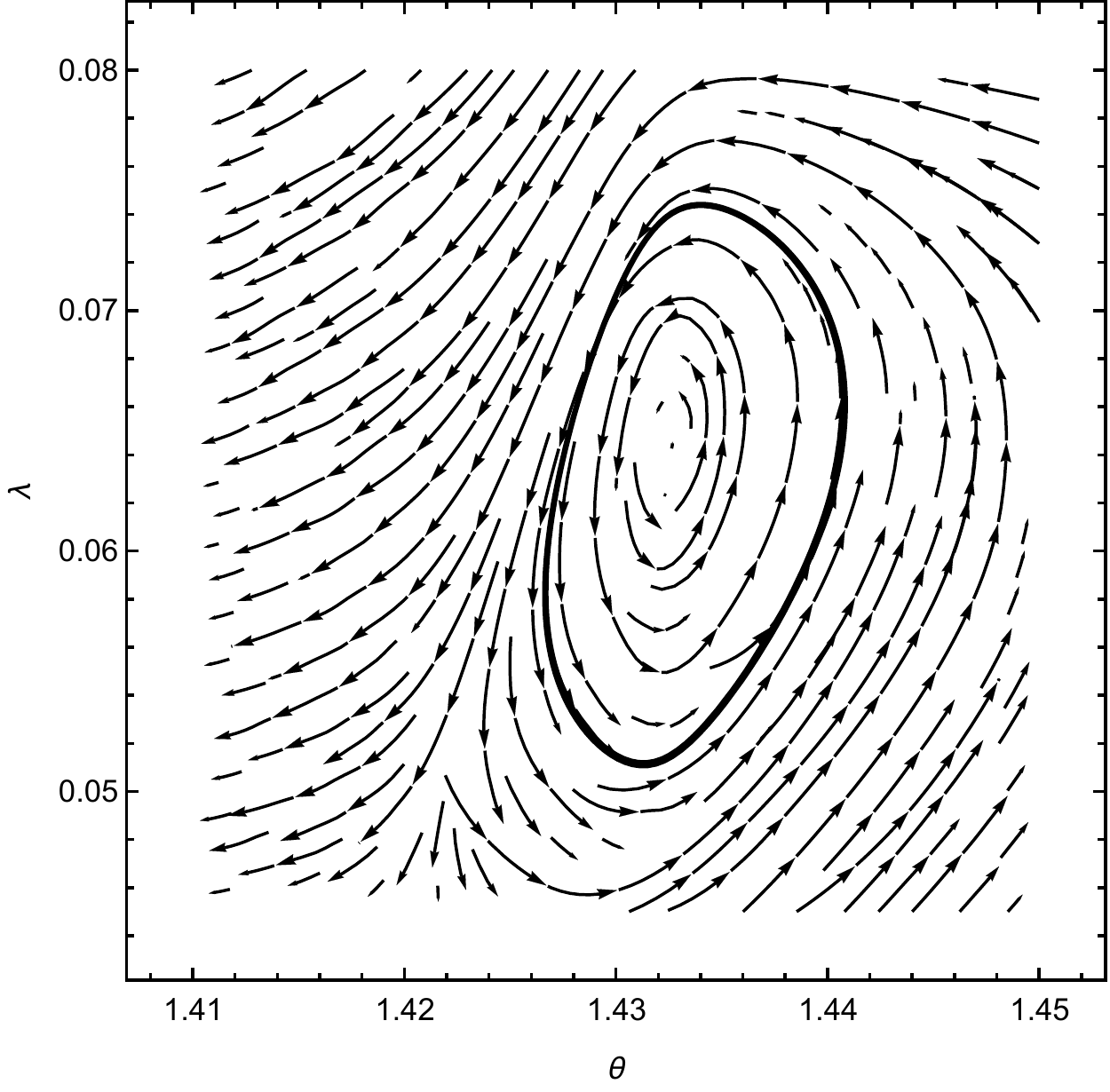}
	\caption{Time series of dimensional temperature and ice sheet extent (left) and the phase plane of their dimensionless versions (right). Parameters used for this simulations are listed in Tab. \ref{tab:Parameters} along with $\mu=1.0545 \mu_0 = 2.0196$ and $\mu_0=1.915$. The limit cycle is drawn in thick while its Lyapunov coefficient is $l_1(\mu_0) = -162.3$. This is the largest limit cycle found numerically.}
	\label{fig:TempIce}
\end{figure}

Looking at Fig. \ref{fig:TempIce} we immediately can see that although the position of the critical point is reasonable, the amplitude of the oscillations is smaller than the one predicted in the original work (\cite{Kal97}). During ice ages, the temperature oscillates through several degrees. The fundamental reason for such behaviour is putting $\epsilon = 0$ at the beginning of our analysis. This forbids the snow line to follow the temperature changes and hence, to increase the temperature amplitude. This can be seen from simple physical reasoning. If the temperature is high, there is a larger snowfall at high latitudes promoting the growth of the ice sheet. Then, the large ice extent increases albedo which, in turn, lowers the temperature. In the case of $\epsilon\neq 0$ the snow line gets lower and shrinks the ablation zone. Eventually, when the precipitation is low enough the melting dominates and ice retreats. Temperature rises forcing the snow line up and the oscillation continues so forth. On the other hand, when $\epsilon = 0$ there is no additional mechanism which shortens the ablation zone and hence, the amplitude of oscillations is smaller than in the other case. 

On the other hand, when $\mu$ passes through $\mu_0$ the amplitude of the oscillations grows until it reaches the size depicted on Fig. \ref{fig:TempIce}. This is in accord with the regime of model's validity - the excursions from the equilibrium should not be large. Moreover, if the parameters were tweaked for the nullclines to intersect precisely at one point, the limit cycle would not be attracted by another critical point. The oscillations would persist in a form of asymmetric relaxations-oscillations. The analysis of this case has been done in \cite{Plo19}.  

We stress, however, the fact that the period of the oscillations is of the right order that is present in the nature, that is $10-100$k years (in our case - $40$k years). Probably, it is possible to get more realistic results by varying the parameters (especially steepness $\Delta \alpha$, $\Delta \xi$ and centres $T_c$, $\xi_c$) but our goal is to show that even for the most simple model of ice-albedo and precipitation-temperature feedbacks there is a stable limit cycle. 

Lastly, we state one more result about the possible type of the critical point. 

\begin{prop}
\label{prop:Tangent}
	Let $(\theta_c,\lambda_c)$ be a critical point of (\ref{eqn:MainSystem}) for which $f'_c=g'_c>0$. Then, if $\mu>\mu_0$ it is unstable while for $\mu<\mu_0$ its stability is the same as the stability of the $\psi=0$ point of the following differential equation
	\begin{equation}
		\frac{d\psi}{d\tau} = \frac{2 \alpha _2 \gamma  \mu \xi_c^2}{ \left(\alpha _2 \gamma  \sqrt{\lambda_c } \mu  f'_c-\xi_c \right)^3}\left(f_c''-g''_c\right)\psi^2 + \text{ higher order terms}.
	\label{eqn:CenterManifoldTheorem}
	\end{equation}
	Therefore, if $f''_c\neq g''_c$ the critical point is unstable. 
\end{prop}
\begin{proof}
First by (\ref{eqn:JacobianEigenvalues}), if $f'_c=g'_c$ the determinant of the Jacobian matrix vanishes one of the eigenvalues is equal to zero. Then one eigenvalue vanishes while the other is equal to $\mu\alpha_2\gamma f'_c-\frac{\xi_c}{\sqrt{\lambda_c}}$. Now, if $\mu>\mu_0$ the latter is positive and hence, the critical point is unstable. 

Assume now that $\mu<\mu_0$. Then, one of the eigenvalues is positive while the other is zero. At this point we cannot infer about its stability and we will use the Central Manifold Theorem (see \cite{Kuz13}). It simply states that on the phase plane there exist a graph of a function to which all sufficiently close trajectories converge exponentially fast. On this manifold, however, the dynamics is slower and governed by a smaller number of differential equations (in our case - one).

To find the center manifold first we have to diagonalize the system (\ref{eqn:MainSystem}) in order to find the direction in which the vector field has a zero eigenvalue. A simple calculation gives the eigenvectors
\begin{equation}
	\textbf{p} = \left(
	\begin{array}{c}
	-\dfrac{\xi_c }{\sqrt{\lambda_c }} \vspace{4pt} \\
	-\dfrac{4 \lambda ^{3/2} \xi_c '}{\xi_c }
	\end{array}
	\right), \quad
	\textbf{q} = \left(
	\begin{array}{c}
	\gamma  \mu  \alpha _1 f_c' \vspace{4 pt} \\
	\dfrac{4 \lambda_c ^{3/2} \xi_c '}{\xi_c }
	\end{array}
	\right),
\end{equation}
and the transition matrix
\begin{equation}
	S = \left(\begin{array}{cc}
		\textbf{p} & \textbf{q}
	\end{array}\right).
\end{equation}
If introduce the following variables
\begin{equation}
		\left(
	\begin{array}{c}
	\psi \\
	\kappa
	\end{array}\right) = S^{-1} 
	\left(
		\begin{array}{c}
		\theta-\theta_c \\
		\lambda-\lambda_c
		\end{array}
	\right),
\end{equation}
the system (\ref{eqn:MainSystem}) becomes
\begin{equation}
\begin{split}
	\frac{d}{dt}
	\left(
	\begin{array}{c}
	\psi \\
	\kappa
	\end{array}\right) &= -\frac{\xi_c }{4 \lambda_c  \xi_c  \xi_c '-4 \alpha _2 \gamma  \lambda ^{3/2} \mu  f' \xi '} 
	\left(
	\begin{array}{c}
		\mu  \left(\frac{4 \lambda_c ^{3/2} \xi_c '}{\xi_c }F(\theta,\lambda)-\alpha _1 \gamma  f'_c G(\theta,\lambda)\right) \\
		\frac{4 \lambda_c ^{3/2} \mu  \xi_c '}{\xi_c }F(\theta,\lambda)-\frac{\xi_c }{\sqrt{\lambda_c }}G(\theta,\lambda) 
	\end{array}
	\right) =:
	\left(
	\begin{array}{c}
	\widetilde{F}(\psi,\kappa) \\
	\widetilde{G}(\psi,\kappa)
	\end{array}\right),
	\\
	&\text{where } (\theta,\lambda) = \left(\theta_c+\alpha _1 \gamma    \mu  f'_c\kappa-\frac{\xi_c }{\sqrt{\lambda_c }}\psi, \lambda_c + \frac{4 \lambda_c ^{3/2}  \xi_c '}{\xi_c }(\kappa -\psi )\right).
\end{split}
\label{eqn:MainSystemTransformed}
\end{equation}
From the Central Manifold Theorem there exist a function $K=K(\psi,\mu)$ which satisfies the following partial differential equation
\begin{equation}
	\widetilde{G}(\psi,K(\psi,\mu),\mu) = K_\psi(\psi,\mu) \widetilde{F}(\psi,K(\psi,\mu),\mu), \quad K(0,0) = K_\psi(0,0) = 0.
\label{eqn:CenterManifoldPDE}
\end{equation}
Then, the dynamics of (\ref{eqn:MainSystemTransformed}) is captured by the equation
\begin{equation}
	\frac{d\psi}{d\tau} = \widetilde{F}(\psi,K(\psi)). 
\label{eqn:CenterManifoldODE}
\end{equation}
Since we are dealing with the local behaviour near critical point we can expand the center manifold into Taylor series (this is the only meaningful coefficient)
\begin{equation}
	K(\psi) = c_2 \psi^2 + \text{ higher order terms}, 
\end{equation}
and plug it to (\ref{eqn:CenterManifoldPDE}) to arrive at 
\begin{equation}
	c_2 = \frac{\alpha _2 \gamma  \sqrt{\lambda_c } \mu  \xi_c ' \left(4 \lambda_c^2 \xi_c ''-\xi_c^2 f_c''\right)+\xi_c ^2 \xi_c '' \left(\xi_c -\alpha _1 \gamma  \sqrt{\lambda_c } \mu  f'_c\right)-8 \lambda_c  \xi_c ^2 \left(\xi_c '\right)^2}{2 \sqrt{\lambda_c } \xi_c ' \left(\xi_c -\alpha _1 \gamma  \sqrt{\lambda_c } \mu  f'_c\right)^2}.
\end{equation}
If we now plug the above expansion of the center manifold into (\ref{eqn:CenterManifoldODE}) we will arrive at
\begin{equation}
	\frac{d\psi}{d\tau} = \frac{\alpha _2 \gamma  \mu  \left(\xi_c ^2 f'_c-4 \lambda_c ^2 \xi_c '\right)}{\xi_c  \left(\xi_c -\alpha _2 \gamma  \sqrt{\lambda_c } \mu  f'_c\right)} \psi + \text{ higher order terms}.
\end{equation}
However, by using (\ref{eqn:Nullclines}) and (\ref{eqn:NullclinesDerivatives}) we have
\begin{equation}
	\xi_c ^2 f_c'-4 \lambda ^2 \xi_c ' = \alpha_2 \gamma \mu \xi_c^2 \left(f'_c-\frac{\xi_c'}{(1+\xi_c)^2}\right) = \alpha_2 \gamma \mu \xi_c^2 \left(f'_c-g'_c\right) = 0,
\end{equation}
by our assumption. Therefore the leading order term in the above differential equation vanishes and we have to take into account higher orders of $\psi$. Using the above calculation and computing the next term in the Taylor series we obtain
\begin{equation}
	\psi' = \frac{2 \alpha _2 \gamma  \mu \left(\xi_c ^3 f_c''+32 \lambda_c ^3 \left(\xi_c '\right)^2-4 \lambda_c ^2 \xi_c  \xi_c ''\right)}{ \xi_c \left(\alpha _2 \gamma  \sqrt{\lambda_c } \mu  f'_c-\xi_c \right)^3}\psi^2 + \text{ higher order terms}.
\end{equation}
This can be further simplified by again using (\ref{eqn:Nullclines}) and noticing that
\begin{equation}
	4 \frac{\lambda_c ^2}{\xi_c^2}  \xi_c '' -32 \frac{\lambda_c ^3}{\xi_c^3} \left(\xi_c '\right)^2 = \frac{1}{4}\frac{\xi''_c}{(1+\xi_c)^2} - \frac{1}{2} \frac{\left(\xi'_c\right)^2}{(1+\xi_c)^2} = g''_c.
\end{equation}
From that there follows the assertion (\ref{eqn:CenterManifoldTheorem}). 
\end{proof}

\section{Conclusion and future work}
We have shown that even a very simple conceptual model describing ice-albedo and precipitation-temperature feedbacks can exhibit internal self-sustained oscillations. The model is by no means constructed to represent the actual state of the climate but rather to describe a simple mechanism which takes into account two important features of Earth system. 

Our results rigorously state that under some natural conditions, the temperature and the ice sheet extent oscillate at a period of $O(10^4 \text{ years})$. The amplitude is, however, smaller than in reality but this behaviour can be anticipated from the very construction. Moreover, the results are completely independent from the particular choice of the functional form of the oceanic albedo and accumulation to ablation ratio (which are hard to determine experimentally). This generality strengthens the fact that the climate may indeed oscillate without any external forcing. 

Our future work will consists of considering the case of nonzero $\epsilon$ which would be a function of time and temperature in order to take into account the temperature-melting feedback of the ice sheets. The ultimate goal would be to augment the model with equations representing lithosphere response and CO$_2$ oscillations.

\section*{Appendix}
The calculation of the first Lypanunov coefficient is a very technical point of the above considerations. We have to use (\ref{eqn:LyapunovFormula}) and in order to do that, calculate the necessary second and third derivatives. For the function $F$ we have
\begin{equation}
\begin{split}
	\frac{\partial^2 F}{\partial \theta^2} &= \alpha_2\gamma f''_c, \quad \frac{\partial^2 F}{\partial \lambda\partial\theta} = \frac{\partial^2 F}{\partial \lambda^2} = 0, \\
	\frac{\partial^3 F}{\partial \theta^3} &= \alpha_2\gamma f'''_c, \quad \frac{\partial^3 F}{\partial \theta^2 \partial\lambda} = \frac{\partial^3 F}{\partial \theta \partial\lambda^2} = \frac{\partial^3 F}{\partial \lambda^3} = 0,
\end{split}
\end{equation}
while the derivatives of $G$ can be written as
\begin{equation}
\begin{split}
	\frac{\partial^2 G}{\partial \theta^2} &= \frac{4 \lambda_c^{3/2}}{\xi_c} \xi_c'', \quad \frac{\partial^2 G}{\partial \lambda\partial\theta} = \frac{2 \sqrt{\lambda_c} (1-2 \xi_c )}{\xi_c }\xi_c', \quad \frac{\partial^2 G}{\partial \lambda^2} = -\frac{\xi_c}{\lambda_c^{3/2}}, \\
	\frac{\partial^3 G}{\partial \theta^3} &= \frac{4 \lambda_c ^{3/2} }{\xi_c}\xi_c''' \quad \frac{\partial^3 G}{\partial \theta^2 \partial\lambda} = \frac{2 \sqrt{\lambda } (1-2 \xi_c)}{\xi_c} \xi_c'', \quad \frac{\partial^3 G}{\partial \theta \partial\lambda^2} = -\frac{(4 \xi_c +1)}{\sqrt{\lambda_c} \xi_c }\xi_c', \quad \frac{\partial^3 G}{\partial \lambda^3} = \frac{3 \xi_c }{4 \lambda_c ^{5/2}}.
\end{split}
\end{equation}
Having those we can compute the derivatives of $P$ and $Q$ in (\ref{eqn:LyapunovFormula}). In order to facilitate this tedious process we have used the Mathematica scientific environment. The result is the following
\begin{equation}
\begin{split}
	l_1(\mu_0) &= \frac{4 f_c^{(3)} \lambda_c ^2 \xi_c ^2+\left(f'_c\right)^2 \left(3 \xi_c ^2 g'_c-8 \lambda_c ^2 (4 \xi_c +1) \xi_c '\right)+8 \lambda_c ^3 (1-2 \xi_c ) f' \xi_c ''}{32 \lambda_c ^2 \xi_c ^2 \left(f'\right)^2 g'_c \sqrt{\frac{g'_c}{f'_c}-1}} \\
	&+ \frac{\lambda ^2 \xi_c ^2 f_c'' \left(4 \lambda_c ^2 \xi_c ''-\xi_c ^2 f_c''\right)+\xi_c ^2 \left(f_c'\right)^3 \left(\xi_c ^2 g'+4 \lambda_c ^2 (2 \xi_c -1) \xi_c '\right)}{8 \lambda_c ^2 \xi_c ^4 \left(f_c'\right)^2 g_c' \left(f_c'-g_c'\right) \sqrt{\frac{g_c'}{f_c'}-1}} \\
	&+ \frac{2 \lambda_c ^2 \left(f_c'\right)^2 \left((2 \xi_c -1) \xi_c ' \left(\xi_c ^2 g_c'+4 \lambda_c ^2 (2 \xi_c -1) \xi_c '\right)-2 \lambda_c  \xi_c ^2 \xi_c ''\right)-2 \lambda_c ^3 (2 \xi_c -1) f_c' \xi_c ' \left(\xi_c ^2 f_c''+4 \lambda_c ^2 \xi_c ''\right)}{8 \lambda_c ^2 \xi_c ^4 \left(f_c'\right)^2 g_c' \left(f_c'-g_c'\right) \sqrt{\frac{g_c'}{f_c'}-1}}.
\end{split}
\label{eqn:LyapunovCoefficient}
\end{equation}
The sign of the above can now be numerically evaluated to infer the stability of a limit cycle. 


\end{document}